\documentclass[11pt]{my_amsart}
\usepackage{amsmath}
\usepackage{amssymb,array}
\usepackage{mathtools}
\usepackage{xfrac}
\usepackage[all,cmtip]{xy}
\usepackage{etoolbox}
\usepackage[title]{appendix}

\newtoggle{omegas}
\togglefalse{omegas}

\allowdisplaybreaks
\begin{document}

\newtheorem{thm}{Theorem}[section]
\newtheorem*{thm*}{Theorem}
\newtheorem{lem}[thm]{Lemma}
\newtheorem*{lem*}{Lemma}
\newtheorem{prop}[thm]{Proposition}
\newtheorem*{prop*}{Proposition}
\newtheorem{cor}[thm]{Corollary}
\newtheorem{defn}[thm]{Definition}
\newtheorem*{conj*}{Conjecture}
\newtheorem{conj}[thm]{Conjecture}
\theoremstyle{remark}
\newtheorem*{remark}{Remark}
\newtheorem*{question*}{Question}
\newtheorem{question}{Question}

\numberwithin{equation}{section}

\newcommand{\Z}{{\mathbb Z}} 
\newcommand{\Q}{{\mathbb Q}}
\newcommand{\R}{{\mathbb R}}
\newcommand{\C}{{\mathbb C}}
\newcommand{\N}{{\mathbb N}}
\newcommand{\NC}{{\mathcal N}}
\newcommand{\PR}{{\mathcal P}}
\newcommand{\FF}{{\mathbb F}}
\newcommand{\fq}{\mathbb{F}_q}
\newcommand{\fqtm}{\mathbb{F}_q[t]^{<m}}
\newcommand{\feq}{\overline{\mathbb F}_q}
\newcommand{\rinf}{\rightarrow \infty}
\newcommand{\partx}{\tfrac{\partial}{\partial x}}
\newcommand{\partt}{\tfrac{\partial}{\partial t}}
\newcommand{\partxf}[1]{\tfrac{\partial #1}{\partial x}}
\newcommand{\parttf}[1]{\tfrac{\partial #1}{\partial t}}

\newcommand{\rmk}[1]{\footnote{{\bf Comment:} #1}}

\renewcommand{\mod}{\operatorname{mod}}
\newcommand{\ord}{\operatorname{ord}}
\newcommand{\TT}{\mathbb{T}}
\renewcommand{\d}{{\mathrm{d}}}
\renewcommand{\^}{\widehat}
\newcommand{\HH}{\mathbb H}
\newcommand{\Vol}{\operatorname{vol}}
\newcommand{\area}{\operatorname{area}}
\newcommand{\tr}{\operatorname{tr}}
\newcommand{\norm}{\mathcal N} 
\newcommand{\intinf}{\int_{-\infty}^\infty}
\newcommand{\ave}[1]{\left\langle#1\right\rangle} 
\newcommand{\Var}{\operatorname{Var}}
\newcommand{\Prob}{\operatorname{Prob}}
\newcommand{\sym}{\operatorname{Sym}}
\newcommand{\disc}{\operatorname{disc}}
\newcommand{\delt}{\operatorname{\delta}}
\newcommand{\tdeg}{\operatorname{tot.deg}}
\newcommand{\pisc}{\operatorname{disc}_{+}}
\newcommand{\Berl}{\operatorname{Berl}}
\newcommand{\hgt}{\operatorname{ht}}
\newcommand{\gal}{\operatorname{Gal}}
\newcommand{\CA}{{\mathcal C}_A}
\newcommand{\cond}{\operatorname{cond}} 
\newcommand{\lcm}{\operatorname{lcm}}
\newcommand{\Kl}{\operatorname{Kl}} 
\newcommand{\leg}[2]{\left( \frac{#1}{#2} \right)}  
\newcommand*\rfrac[2]{{}^{#1}\!/_{#2}}

\newcommand{\sumstar}{\sideset \and^{*} \to \sum}

\newcommand{\LL}{\mathcal L} 
\newcommand{\sumf}{\sum^\flat}
\newcommand{\Hgev}{\mathcal H_{2g+2,q}}
\newcommand{\USp}{\operatorname{USp}}
\newcommand{\conv}{*}
\newcommand{\dist} {\operatorname{dist}}
\newcommand{\CF}{c_0} 
\newcommand{\kerp}{\mathcal K}

\newcommand{\fs}{\mathfrak S}
\newcommand{\rest}{\operatorname{Res}} 
\newcommand{\af}{\mathbb A} 

\newcommand{\lbb}{\left[}
\newcommand{\rbb}{\right]}
\newcommand{\lb}{\left(}
\newcommand{\rb}{\right)}
\newcommand{\al}{\alpha}
\newcommand{\eps}{\epsilon}
\newcommand{\ze}{\zeta}
\newcommand{\lam}{\lambda}
\newcommand{\ity}{\infty}

\DeclarePairedDelimiter{\floor}{\lfloor}{\rfloor}
\DeclarePairedDelimiter{\ceil}{\lceil}{\rceil}

\author[Dan Carmon]{Dan Carmon\\With appendix by Alexei Entin}
\title[Square-free values of large polynomials]
{On square-free values of large polynomials 
over the rational function field}

\address{Raymond and Beverly Sackler School of Mathematical Sciences,
Tel Aviv University, Tel Aviv 69978, Israel}
\date{\today}
\thanks{The research leading to these results has received funding from the European
Research Council under the European Union's Seventh Framework Programme
(FP7/2007-2013) / ERC grant agreement n$^{\text{o}}$ 320755.}

\begin{abstract}
 We investigate the density of square-free values of polynomials with large
 coefficients over the rational function field $\fq[t]$. Some interesting
 questions answered as special cases of our results include the density of
 square-free polynomials in short intervals, and an asymptotic for the
 number of representations of a large polynomial $N$ as a sum of a small
 $k$-th power and a square-free polynomial.
 
\end{abstract}

\maketitle 

\section{Introduction - classical open problems}\label{intro}
In this paper we establish function field analogues to certain classical open
problems in analytic number theory, such as the representation of large 
integers by a sum of a square-free integer and a $k$-th power. We replace
the large integers, by way of analogy, with polynomials of large degree over 
a fixed finite field $\fq$. Our new results will be presented
in section \ref{new res sec}. We shall first review the classical problems
whose analogues we investigate, as well as currently known partial or 
conditional results about these questions.

\subsection{Square-free values of polynomials}
An integer $n$ is called {\it square-free} if it is not divisible by the square of any integer $d > 1$.
It is well known that the ``probability'' of a large ``random'' integer to be square-free is $\frac1{\zeta(2)}$
-- more precisely, this is the density of the set of square-frees in the positive integers.
A classical problem in number theory concerns the density of square-free values of polynomials:
\begin{question}
 Let $f \in \Z[x]$ be a polynomial of degree $k$. Are there infinitely many positive integers $n$ such that $f(n)$ is square-free?
 More ambitiously, compute the density of $\{n \in \N : f(n) \text{ square-free}\}$.
\end{question}

There are two obvious obstructions for such an $f$ being square-free infinitely often. If $f$ is divisible by the square of some
non-constant polynomial $g \in \Z[x]$, then clearly $f(a)$ can only be square-free when $g(a) = \pm 1$, which occurs for only
finitely many $a$ - this is a {\it global} obstruction. On the other hand, if for some prime $p$, $f(a)$ is divisible 
by $p^2$ for every $a$, then clearly $f(a)$ is never square-free. These are the {\it local} obstructions, as they depend only 
on the behaviour of $f$ modulo prime powers.

Define for any $d > 1$,
$$\rho(d) := \#\{a \mod d : f(a) \equiv 0 \!\!\!\pmod d\}.$$
For small primes $p$, the probability that $f(a)$ is not divisible by $p^2$
is approximately $1 - \frac{\rho(p^2)}{p^2}$. Heuristically, one expects these events to be
nearly independent, hence the probability that $f(a)$ is indivisible by
$p^2$ for all primes $p$ should be $\prod_{p \in \PR} \left(1 - \frac{\rho(p^2)}{p^2}\right)$,
where $\PR$ is the set of primes. Note that not being divisible by any $p^2$ is 
equivalent to being square-free. This leads to the following conjecture:

\begin{conj}\label{sqfree values conj}
Let $f \in \Z[x]$ be a square-free polynomial of degree $k$. 
The set $\{n \in \N : f(n) \text{ square-free}\}$ is conjectured to have density
$$c_f := \prod_{p \in \PR} \left(1 - \frac{\rho(p^2)}{p^2}\right).$$
\end{conj}
Note that if there is a local obstruction at a prime $p$, then $\rho(p^2) = p^2$ and the above product is $0$. Otherwise,
it is easily seen that $\rho(p^2) \le k$ for $p$ sufficiently large, hence the infinite product converges, and $c_f$
is positive.

For $k=1$, the conjecture is equivalent to the regular density of the square-frees.
The conjecture has been proved for $k=2$ by
Ricci in the 1930's \cite{Ricci}, and for $k=3$ by Hooley in 1968 \cite{Hooley}. 
Unconditionally, the conjecture
remains completely open for all $k \ge 4$. However, in \cite{Granville},
Granville proved the conjecture in full generality, assuming the $ABC$ conjecture.

\subsection{A dual problem}
In the previous section we considered the density of square-free values of a fixed polynomial with fixed coefficients,
as the argument grows larger and larger. What happens if we allow the polynomial to vary, with coefficients tending to infinity,
possibly faster than the arguments? An example of this kind of question is the following:
\begin{question}\label{quest sums}
 Does every sufficiently large $N \in \Z$ admit a representation as a sum $N = x^k + r$ of a positive $k$-th power and a positive square-free?
 How many such representations are there, asymptotically?
\end{question}
Clearly, finding such representations is equivalent to finding $x < N^{1/k}$ with $f(x) = N - x^k$ being square-free.
Hence by the same heuristic arguments as before, we might expect the answer to be $c_f N^{1/k}$, where $c_f$ is defined
precisely the same -- however, note that it now depends on $N$, as does $f$ itself. As such, Question \ref{quest sums}
does not follow immediately from Conjecture \ref{sqfree values conj}, although it might be resolved by similar techniques.
Question \ref{quest sums} has been answered positively for $k=2$ by Estermann in \cite{Estermann}. 
The case $k=3$ was stated by Hooley \cite[\S4.6, Theorem 4]{HooleyBook}\footnote{In
the form that any sufficiently large number
is the sum of a cube and a square-free integer,
with no claim on the asymptotic number of representations.}.
Question \ref{quest sums} appears more difficult and nuanced than 
Conjecture \ref{sqfree values conj}. Indeed, the proof outline Hooley presents for
$k=3$ uses strictly more ideas and methods than his proof of the density
of the square-free values of cubics -- and still cannot obtain the number of representations,
nor is it applicable when $x^3$ is replaced with a general cubic
polynomial. It is thus unsurprising that the case $k \ge 4$ is still open.

\subsection{Square-frees in short intervals}
Another classical problem of interest regards the number of square-free integers
in short intervals, i.e. sets of the form $I(X,H) = \{n \in \Z : X \le n < X+H\}$,
where $H$ is much smaller than $X$. Clearly, the expectation of the
density of square-free integers in such sets, when we average over all $X$,
should be the same as that over all integers, i.e. $\tfrac{1}{\zeta(2)}$. We are
interested in understanding how small we may take $H$, as a function of the size
of $X$, such that the density  will be accurate (up to smaller order deviations) for all $X$,
and not just on average or for almost all $X$. This gives rise to the following classical conjecture:
\begin{conj}\label{short interval conj}
 Let $\epsilon > 0$ be fixed, let $X$ be large, and let $H \gg  X^\epsilon$. Then
 $$\#\{n \in I(X,H) : n \text{ is square-free} \} = (1+o(1))\frac{H}{\zeta(2)}.$$
\end{conj}
Again, this conjecture follows from the $ABC$ conjecture by Granville's method -- 
see the Appendix; Granville \cite{Granville} showed that the $ABC$ conjecture
implies that for any fixed $\epsilon > 0$ there exist square-free integers
in $I(X,X^\epsilon)$, for all $X$ sufficiently large.
Unconditionally, the best known result is due to Tolev \cite{Tolev},
who proved the asymptotic for any $H = H(X)$ such that $\tfrac{H}{X^{1/5}\log(X)} \rinf$,
building on results of Filaseta and Trifonov \cite{Filatri}.

More ambitiously, we may ask this question not only for square-free integers, but 
for square-free values of polynomials:
\begin{question}
 Let $f \in \Z[x]$ be a square-free polynomial with $c_f > 0$. How small
 may we take $H = H(X)$ such that the asymptotic
 $$\#\{n \in I(X,H) : f(n) \text{ is square-free}\} \sim c_f H$$
 will hold for all $X$?
\end{question}

\section{Function field analogues}

\subsection{Square-free values of polynomials in function fields}
One may ask similar questions about polynomials over function fields, rather than
over the integers. Fix a prime power $q$, let $\fq$ be the finite field with $q$
elements, and let $A = \fq[t]$ be the ring of polynomials over $\fq$. Let 
$f \in A[x]$ be a square-free polynomial of degree $k$ (in $x$). As before, one
may ask for the density of the square-free values of $f$. As in the integers,
one may heuristically assume that the contributions from different primes in $A$
are independent, and conjecture a density based on that heuristic. 
It turns out that in this setting, one may actually prove that this density is correct:

\begin{thm}
 Let $\PR$ be the set of primes in $A$ (i.e. monic, irreducible polynomials). For any $D \in A$, let
 $\rho(D) := \#\{a \mod D : f(a) \equiv 0 \pmod D\}$, $||D|| := \#\{a \mod D \} = q^{\deg D}$,
 and $c_f := \prod_{P \in \PR} \left(1 - \frac{\rho(P^2)}{||P||^2}\right)$. Then
 $$\#\{a \in \fq[t] : \deg_t(a) < m, f(a) \text{ square-free}\} = c_fq^m + o(q^m)$$
 as $m$ tends to $\infty$.
\end{thm}
This theorem was first proved by Ramsay \cite{Ramsay}; however, his proof was valid
only for polynomials $f \in \fq[x]$, rather than $\fq[t,x]$, i.e. only polynomials
with constant coefficients. Poonen \cite{Poonen} proved the theorem for all $\fq[t,x]$, 
and generalized it further to multivariate polynomials in $\fq[t,x_1,\dots,x_n]$.
In his 2014 M.Sc. thesis, Lando \cite{Lando} gave a quantitative version of Poonen's
work, and applied it to the problems of square-free and power-free
values at prime polynomials. 

\subsection{New results}\label{new res sec}
Our main goal in this paper is to extend the above results to polynomials $f$
with large coefficients, giving quantitative answers to questions analogous to
those presented in section \ref{intro},
after replacing the integers with the polynomial ring over $\fq$.
Our methods include carefully applying Poonen's and Lando's techniques, as well as
replacing some na\"\i ve sieving arguments with the more sophisticated Brun sieve. 
Specifically, we show:

%
%

\begin{thm} \label{brun thm}
 Let $q=p^e$ be a fixed prime power, let $k > 0$ be a fixed integer, 
 and let $m,n$ be positive integers
\iftoggle{omegas}
 {with $m = \omega(\log_qn \log_q\log_qn)$.}
 {with $m \gg \log_qn \log_q\log_qn$ and $m \rinf$.\footnote{Note
 that if $n$ is bounded, Theorem \ref{brun thm} is equivalent to
 Poonen's theorem. We would therefore be interested mostly
 in the case $n \rinf$, and $m \rinf$ would follow from
 the bound $m \gg \log_qn \log_q\log_qn$. In the course of the proof,
 we provide an explicit upper bound on the constant.}}
 Let $f \in \fq[t,x]$ be a square-free polynomial 
 with $\deg_x f \le k$, $\deg_t f \le n$. 
 Let $c_f$ be defined as before.
  Then
 \begin{align*}
  \#\{a \in \fq[t] : \deg a < m, f(a) \text{ square-free}\} = 
  c_fq^m (1+o(1)).
 \end{align*}
 
\end{thm}

From which we may immediately derive an analogue of Question \ref{quest sums}:

\begin{cor}
 Let $q=p^e$ be a fixed prime power, let $k > 0$ be a fixed integer, and let 
 $N \in \fq[t]$ be of sufficiently large degree $n$.
 Additionally, suppose that either $k$ is co-prime to $p$, or $N$ is not a $p$-th power.
 Then $N$ has 
 $c_{N,k} q^{\ceil{n/k}}(1+o(1))$ representations as $N = x^k + r$, with $x,r \in \fq[t]$ 
 such that $r$ is square-free and $\deg x < \tfrac{n}{k}$, where 
 $c_{N,k} = \prod_{P \in \PR} \left(1 - \frac{\rho_{N,k}(P^2)}{P^2}\right)$ and
 $\rho_{N,k}(D) = \#\{a \mod D : a^k \equiv N \pmod D\}.$
\end{cor}

Indeed, this is exactly the number of square-free
values of $f(x) = N - x^k$,\linebreak which is square-free\footnote{Note that if $N$ is a $p$-th
power and $p \mid k$, then $f(x)$ is a $p$-th power as well, hence not square-free.
It is easy to see that in all other cases, $f$ is square-free, by considering
its derivatives by $x$ and $t$, which are co-prime to $f$ whenever they are non-zero.}
and has $\deg_x f = k$, $\deg_t f = n$,
where $x$ ranges over
polynomials of degree less than $m = \ceil{\tfrac{n}{k}}$, which clearly satisfies the assumptions
of Theorem \ref{brun thm} as $n \rinf$. 

If we were to apply Theorem \ref{brun thm} to a short interval setting,
with an interval of length $H = q^m$
consisting of polynomial of size $X = q^n$, it would state that we have the correct asymptotic
\iftoggle{omegas}
{under $m = \omega(\log_qn \log_q\log_qn)$, which is equivalent to
$H = (\log_q X)^{\omega(\log_q\log_q\log_q X)}$.}
{when $m \gg \log_qn \log_q\log_qn$,
or equivalently, for $H \ge (\log_q X)^{C \log_q\log_q\log_q X}$ for
a certain constant $C$ and all sufficiently large $X$.} 
This is already much weaker than the condition $H \gg X^\epsilon$, 
but in fact we can go even lower:

\begin{thm} \label{short interval thm}
 Let $q = p^e$ be a fixed prime power, and $g \in \fq[t,x]$ a fixed square-free
 polynomial with $\deg_x g = k$. Let $n,m$ be large positive integers with 
 \iftoggle{omegas}
 {$m = p(\log_q n - \log_q \log_q n) + \omega(1)$}
 {$m - p(\log_q n - \log_q \log_q n) \rinf$}, and let
 $N(t) \in \fq[t]$ be of degree $n$. Consider the interval of size $H = q^m$ around $N$,
 $$I(N,m) = \{N + a : a \in \fq[t], \deg a < m\}.$$
 
 Then 
 \begin{align*}
  \#\{a \in I(N,m) : g(a) \text{ square-free}\} = 
  c_gq^m (1+o(1)). 
 \end{align*}
  
\end{thm}
In terms of $H$ and $X$, the relation 
\iftoggle{omegas}
{$m = p(\log_q n - \log_q \log_q n) + \omega(1)$} 
{$m - p(\log_q n - \log_q \log_q n) \rinf$} translates to
\iftoggle{omegas}
{$H = \omega\left(\!\left(\frac{\log_q X}{\log_q \log_q X}\right)^p \right)$,}
{
$H \ge C\left(\frac{\log_q X}{\log_q \log_q X}\right)^p$ for any constant $C > 0$ 
and all sufficiently large $X$,}
i.e. a polylogarithmic relation.
It seems quite peculiar that the characteristic of the field should play such an important role
in this relation. We remark further that one may find intervals
with $H \gg \frac{\log_q X}{\log_q \log_q X}$ that contain no 
square-free polynomials at all, by a straight-forward application
of the Chinese Remainder Theorem; so this result is nearly sharp.

The proofs of the two theorems are very similar -- they both involve essentially the same
computations, but the different settings lead to different error terms being dominant,
hence different lower bounds on $m$. In fact the two contributions are mostly disjoint, 
which allows us to generalize the two results into one unified theorem:

\begin{thm} \label{unif thm}
  Let $q=p^e$ be a fixed prime power, $k > 0$ a fixed integer, and $m,n_1,n_2$ be varying
  positive integers with both 
  \iftoggle{omegas}
  {$m = \omega(\log_qn_1 \log_q\log_qn_1)$ and
  $m = p(\log_q n_2 - \log_q \log_q n_2 + 2k\log_q\log_qn_1) + \omega(1)$.}
  {$m \gg \log_qn_1 \log_q\log_qn_1$ and
  $m - p(\log_q n_2 - \log_q \log_q n_2 + 2k\log_q\log_qn_1) \rinf$.}
  Let $g \in \fq[t,x]$ be a square-free polynomial 
  with $\deg_x g \le k$, $\deg_t g \le n_1$.
  Let $N(t) \in \fq[t]$ be of degree $n_2$, and let $I(N,m)$ be the interval of size $q^m$ around $N$.
  Then 
 \begin{align*}
  \#\{a \in I(N,m) : g(a) \text{ square-free}\} = 
  c_g q^m (1+o(1)). 
 \end{align*}
\end{thm}

\section{Proof of main theorem}
We will begin by working in the setting of Theorem \ref{brun thm}, for simplicity, but most
of the computations will be immediately applicable to the other theorems as well.
For brevity, let us denote for any set of polynomials $A$ and any degree $d$,
$A^{<d} = \{a \in A : \deg a < d\}$,
and similarly define $A^{\ge d},A^{=d}$.

Let us write $N = \{a \in \fqtm : f(a) \text{ square-free}\}$. The first step towards estimating $\#N$ is
to bound it from below and above by terms more closely related to the contributions of certain primes. We define
\begin{align}
   \label{N' def} N' &= \{a \in \fqtm : \forall P \in \PR^{<m_0}, P^2 \nmid f(a)\} \\
   \label{N'' def} N'' &= \{a \in \fqtm : \exists P \in \PR^{\ge m_0} \cap \PR^{< m_1}, P^2 \mid f(a)\} \\
   \label{N''' def} N''' &= \{a \in \fqtm : \exists P \in \PR^{\ge m_1}, P^2 \mid f(a)\}
\end{align}
where $m_0$ and $m_1$ are appropriately chosen thresholds. Specifically, we take $m_1 = \ceil{m/2}$,
and $m_0$ will be chosen later.

Clearly $N \subseteq N' \subseteq N \cup N'' \cup N'''$, hence $\#N' - \#N'' - \#N''' \le \#N \le \#N'$.
We would therefore like to show that $\#N' = c_fq^m(1+o(1))$ and $\#N'',\#N''' = o(c_f q^m)$. 
Before we proceed to prove these estimates, we need to establish bounds 
for certain sums and products related to $f$.

\subsection{Bounds on the singular sum} \label{singular sect}

We define the {\it singular sum} of the polynomial $f$ as
$S = \sum_{P \in \PR} \frac{\rho(P^2)}{||P||^2}$. We also denote 
the tail of this series by $S(m_0) = \sum_{P \in \PR^{\ge m_0}} \frac{\rho(P^2)}{||P||^2}$.
Our goal in this section is to prove the following bounds on $S, S(m_0)$ and $c_f$:
\begin{lem}
 Let $q$ be a fixed prime power, $k > 0$ a fixed integer, and $n,m_0$ varying integers 
 with $n \rinf$. Let $f \in \fq[t,x]$ be a square-free polynomial 
 with $\deg_x f \le k$ and $\deg_t f \le n$.
 Define $\rho(D), S, S(m_0), c_f$ as above. We have the following asymptotic inequalities:
 \begin{align}
 \label{S bound} S &\le k \ln \log_q n + O(1)  = O(\ln \ln n),  \\
 \label{Sm0 bound} S(m_0) &= O\left(\frac{n}{m_0 q^{m_0}}\right), \\
 \label{c_f bound}\iftoggle{omegas}
 {c_f &= \omega((\log_qn)^{-2k}).}
 {c_f &\gg (\log_qn)^{-2k}.}
\end{align}
\end{lem}

Write $f(t,x) =  f_i(t,x)f_s(t,x)$ where $f_i(t,x) \in \fq[t,x^p]$ is the product of
all irreducible factors of $f(t,x)$ which are inseparable in $x$,
and $f_s(t,x)$ has no $x$-inseparable factors.
From the fact that $f(t,x)$ is square-free, we immediately see that $f_i,f_s$ are co-prime
and square-free, and furthermore $f_i$ is co-prime to $\parttf{f_i}$
and $f_s$ is co-prime to $\partxf{f_s}$: Indeed, if $P(t,x)$ is an irreducible common
divisor of $f_s$ and $\partxf{f_s}$, it is easy to see that either $P^2 \mid f_s$, 
which contradicts $f_s$ being square-free, or else $P \mid \partxf{P}$, which then
implies that $P$ is inseparable in $x$ -- contradicting the fact that $f_s$ has
no inseparable factors. Similarly, if $P(t,x)$ is an irreducible common divisor
of $f_i, \parttf{f_i}$, then again either
$P^2 \mid f_i$, which leads to contradiction, or $P$ is inseparable in $t$. 
Since both $f_i, \parttf{f_i}$ are in $\fq[t,x^p]$, either $P^p$ must also
be a common divisor, contradicting square-freedom, or $P$ is also in $\fq[t,x^p]$.
But since it is also inseparable in $t$, it follows that $P \in \fq[t^p,x^p]$,
which means that $P$ is a $p$-th power, contradicting its irreducibility. 

Now, define $R(t) = \rest_x(f_i,\parttf{f})\rest_x(f_s,\partxf{f}) \in \fq[t]$.
Note that $R(t)$ is non-zero: Indeed, by the above claims,
$\parttf{f} = f_s\parttf{f_i} + f_i \parttf{f_s}$ is co-prime to $f_i$,
and $\partxf{f} = f_i\partxf{f_s} + f_s \partxf{f_i}$ is co-prime to $f_s$.
Note that the $x$- and $t$-degrees of the polynomials $f_i,f_s$ and their derivatives
are all at most $k$ and $n$, respectively. Therefore, both resultants can be
given as polynomials of degree at most $2k$ in the 
$\fq[t]$-coefficients of their arguments, each of which is of degree at most $n$.
Therefore $\deg R \le 4kn = O(n)$. In particular $R$ has at most $\frac{4kn}{m_0}$ prime factors 
of degree at least $m_0$. 

\hspace{-1pt}For any prime $P \in \PR$ such that $P \nmid R$, the residue $f \mod P \in (\fq[t]/(P))[x]$ is non-trivial 
(as every prime dividing the content of $f$ also divides $R$). The residue also has degree $\le k$, 
which then implies $\rho(P) \le k$. 
Let $a \in \fq[t]$ represent a residue class in $\rho(P)$, i.e. 
satisfy $f(a) \equiv 0 \pmod P$. If furthermore
$\frac{\partial f}{\partial x}(a) \not \equiv 0 \pmod P$, then by Hensel's lemma
there is a unique lifting of $a$ to a residue $\tilde a \mod P^2$ satisfying 
$\tilde a \equiv a \pmod P, f(\tilde a) \equiv 0 \pmod {P^2}$. 

If, on the other hand, $\frac{\partial f}{\partial x}(a) \equiv 0 \pmod P$,
then $P$ does not divide $f_s(a)$: 
Otherwise, $a$ is a common root of $f_s$ and $\partxf{f}$ modulo $P$,
which then implies $P \mid \rest_x(f_s,\partxf{f})$, contradicting $P \nmid R$.
From $P \mid f(a) = f_s(a)f_i(a)$ it then follows that $P \mid f_i(a)$, and by
the same argument as above, we must then have
$\parttf{f}(a) \not \equiv 0 \pmod P$, and thus
$$\tfrac{\d f(t,a(t))}{\d t} = \parttf{f}(a)+\partxf{f}(a)\tfrac{\d a}{\d t} 
\equiv \parttf{f}(a) \not \equiv 0 \pmod P.$$
In particular, it follows that $P(t)^2 \nmid f(t,a(t))$, for any such $a$.
Therefore no residue $\tilde a \mod P^2$ with $\tilde a \equiv a \pmod P$
satisfies $f(\tilde a) \equiv 0 \pmod {P^2}$. 

We have shown that for every residue $a \mod P \in \rho(P)$, there is at most one 
lifting modulo $P^2$ which is in $\rho(P^2)$, assuming $P \nmid R$.
Therefore for such primes, $\rho(P^2) \le \rho(P) \le k$.

The contribution of these primes to $S$ is thus at most
\begin{align*}
 \sum_{P \in \PR : P \nmid R} \frac{\rho(P^2)}{||P||^2} &\le \sum_{P \in \PR} \frac{k}{||P||^2} =
  \sum_{d=1}^{\infty} \sum_{P \in \PR^{= d}} \frac{k}{q^{2d}} \\
  & \le \sum_{d=1}^{\infty} \frac{k}{q^{2d}}\frac{q^d}{d} 
  = k \sum_{d=1}^{\infty} \frac{1}{dq^d} \le \frac{k}{q-1} = O(1),
\end{align*}
and similarly their contribution to the tail $S(m_0)$ is at most
\begin{align}\nonumber
 \sum_{P \in \PR^{\ge m_0} : P \nmid R} \frac{\rho(P^2)}{||P||^2} 
 \le k \sum_{d=m_0}^{\infty} \frac{1}{dq^d} 
 = O\left(\frac{1}{m_0q^{m_0}}\right). 
\end{align}

On the other hand, for any prime $P \mid R$, we have $\rho(P^2) \le k||P||$. Indeed, if $P$ divides the content of $f$, then
$\tfrac{f}{P} \in (\fq[t]/(P))[x]$ is non-trivial, as $f$ is square-free and in particular $P^2 \nmid f$. Thus 
\begin{align*}
 \rho(P^2) &= \#\{a \mod P^2 : f(a) \equiv 0\!\!\pmod {P^2}\} \\
 &=\#\{a \mod P^2 : \tfrac{f(a)}{P} \equiv 0 \pmod P\} \\
 &=\#\{a \mod P : \tfrac{f(a)}{P} \equiv 0 \pmod P\}\cdot||P|| \le k||P||,
\end{align*}
while for primes $P \mid R$ that do not divide the content, we simply have $\rho(P) \le k$ and therefore
$\rho(P^2) \le ||P||\rho(P) \le k||P||$.\footnote{A sharper argument shows that 
for primes $P \mid R$ that do not divide the content, we in fact have 
$\rho(P^2) \le \tfrac{k}{2}||P||$, as any root of $f$ modulo $P$ that lifts to $||P||$ roots
modulo $P^2$ must be a double root modulo $P$, and there can be only $k/2$ distinct double roots 
modulo $P$. This allows us to slightly improve the lower bound on $c_f$ for content-free polynomials,
but not in general.
}

Therefore the contribution of the primes $P \mid R$ to the sum $S(m_0)$ is at most
\begin{align*} \nonumber
 \sum_{\substack{P \in \PR^{\ge m_0} \\ P \mid R}} \frac{\rho(P^2)}{||P||^2} \
 &\le \sum_{\substack{P \in \PR^{\ge m_0} \\ P \mid R}} \frac{k||P||}{||P||^2}
 = \sum_{\substack{P \in \PR^{\ge m_0} \\ P \mid R}} \frac{k}{||P||} \\ 
 &\le \sum_{\substack{P \in \PR^{\ge m_0} \\ P \mid R}} \frac{k}{q^{m_0}} 
 \le \frac{4kn}{m_0}\frac{k}{q^{m_0}} = O\left(\frac{n}{m_0q^{m_0}}\right).
\end{align*}

In order to obtain a bound on their contribution to $S$, denote for all $d > 0$,
$u_d = \#\{P \in \PR^{=d} : P \mid R\}$, and let $x_d = d u_d$. The contribution to $S$ is
\begin{align*} \nonumber
 \sum_{\substack{P \in \PR\\ P \mid R}} \frac{\rho(P^2)}{||P||^2} \
 &\le \sum_{\substack{P \in \PR \\ P \mid R}} \frac{k}{||P||} 
 = k \sum_{d=1}^{\infty} \frac{u_d}{q^d} 
 = k \sum_{d=1}^{\infty} \frac{x_d}{dq^d}.
\end{align*}
Note that for all $d >0$, $d u_d \le d\pi_q(d) \le q^d$, and 
$\sum_{d=1}^{\infty} du_d \le \deg R \le 4kn$. As the sequence $\frac{1}{dq^d}$ is decreasing,
it follows that the maximum of $\sum_{d=1}^{\infty} \frac{x_d}{dq^d}$ under the constraints
$0 \le x_d \le q^d$, $\sum_{d=1}^{\infty} x_d \le 4kn$ is attained when $x_d = q^d$ for all $d < n_0$,
$x_{n_0} = 4kn - \sum_{d=1}^{n_0 - 1} x_d$, and $x_d = 0$ for all $d > n_0$. Note that $n_0$ is 
then determined uniquely by $0 \le x_{n_0} \le q^{n_0}$. Such values
would not necessarily correspond to any actual $R$,
but will serve for obtaining an upper bound. It follows that
$q^{n_0 - 1} \le 4kn$, hence $n_0 \le \log_q(4kq n) = \log_q(n) + O(1)$. Thus
\begin{align*}
 \sum_{\substack{P \in \PR\\ P \mid R}} \frac{\rho(P^2)}{||P||^2} 
 \le k \sum_{d=1}^{\infty} \frac{x_d}{dq^d} \le k \sum_{d=1}^{n_0} \frac{1}{d} 
 & = k (\ln(n_0) + O(1)) \\
 & = O(\ln\ln n). 
\end{align*}

It is quite clear that for both $S,S(m_0)$, the bounds for the 
contributions of $P \mid R$ dominate
those of $P \nmid R$, and yield the bounds \eqref{S bound},\eqref{Sm0 bound}.

We now derive the lower bound $c_f \gg (\log_q n)^{-k-o(1)}$
using the upper bound on $S$. Let $\epsilon > 0$, and split the summands
of $S$ into those greater and lesser than $\epsilon$. As each term is at most $\frac{k}{||P||}$,
it follows that only boundedly many are greater than $\epsilon$, and they of bounded degree, thus the contribution of these terms
to the product $c_f = \prod_{P \in \PR}\left(1-\frac{\rho(P^2)}{||P||^2}\right)$ would be bounded below by 
some positive constant $C_\epsilon = C_{k,q,\epsilon} > 0$ independent of $n$ (assuming no local obstructions
exist, so that $1-\frac{\rho(P^2)}{||P||^2} \ge \frac{1}{||P||^2}$  for all $P$). On the other hand,
for summands such that $x = \frac{\rho(P^2)}{||P||^2} < \epsilon$, we have the inequality
$\ln(1-x) > -\frac{x}{1-\epsilon}$, and hence the contributions of such terms to the product $c_f$
is bounded below by $\exp\left(-\frac{S}{1-\epsilon}\right) \gg_{k,q} (\log_q n)^{-k/(1-\epsilon)}$. Taking the
two terms together then yields $c_f \gg_{k,q} C_{\epsilon}(\log_q n)^{-k+O(\epsilon)}$. As $C_\epsilon$ is 
independent of $n$, letting $\epsilon \rightarrow 0$ sufficiently slowly as
$n \rinf$ would allow us to replace the bound by the aforementioned 
$c_f \gg (\log_q n)^{-k-o(1)}$. However, the exact exponent will have negligible 
relevance to our computations, and the bound \eqref{c_f bound} obtained
by choosing $\epsilon =\tfrac12$ suffices for most purposes.
 


%
%

\subsection{Bounding $N''$: Medium primes} \label{N'' sect}

The bound on the medium primes is the easiest of the three, 
and follows immediately from a simple union bound.
Indeed, $m_1$ is chosen such that for any prime $P \in \PR^{< m_1}$ we have $\deg(P^2) < m$ 
and thus
$\#\{a \in \fqtm : P^2 \mid f(a)\} = \frac{\rho(P^2)}{||P||^2}q^m$. Therefore
\begin{align*}
 \#N'' &= \#\{a \in \fqtm : \exists P \in \PR^{\ge m_0} \cap \PR^{< m_1}, P^2 \mid f(a)\} \\
 &= \#\bigcup\nolimits_{P \in \PR^{\ge m_0} \cap \PR^{< m_1}}\{a \in \fqtm : P^2 \mid f(a)\} \\
 &\le \sum\nolimits_{P \in \PR^{\ge m_0} \cap \PR^{< m_1}} \#\{a \in \fqtm : P^2 \mid f(a)\} \\
 &= \sum\nolimits_{P \in \PR^{\ge m_0} \cap \PR^{< m_1}} \frac{\rho(P^2)}{||P||^2}q^m \\
 &\le q^m \sum\nolimits_{P \in \PR^{\ge m_0}} \frac{\rho(P^2)}{||P||^2} 
 = q^m S(m_0).
\end{align*}

It now suffices to choose $m_0$ large enough so that $S(m_0) = o(c_f)$. By \eqref{Sm0 bound},
\eqref{c_f bound}, we see that we may take any $m_0$ such that 
\iftoggle{omegas}
{$m_0 q^{m_0} = \omega(n (\log_qn)^{2k})$}
{$\frac{m_0 q^{m_0}}{n (\log_qn)^{2k}} \rinf$}, 
which is clearly satisfied when e.g.
\iftoggle{omegas}
{$m_0 = \log_qn + 2k\log_q\log_q n + \omega(1)$}
{$m_0 - \log_qn - 2k\log_q\log_q n \rinf $}.
For simplicity, we shall write this condition as $m_0 \gg \log_q n$:
For $n \rinf$, the implied constant may be any constant greater than 1, and
if $n$ is bounded we only require $m_0 \rinf$. 

\subsection{Bounding $N'$: Small primes} \label{N' sect}
We write $\PR(m_0) = \prod_{P \in \PR^{<m_0}} P$.
A standard sieve theory argument gives 
\begin{equation*}\label{N' comp}
 \#N' = \sum_{D \mid \PR(m_0)} \mu(D)\#\{a \in \fq[t]^{<m} : D^2 \mid f(a)\}.
\end{equation*}
For any square-free polynomial $D \in \fq[t]$, let $\nu(D)$ be the number of its prime factors.
For a non-negative integer $k$, define
$$n_k = \sum_{{\substack{D \mid P(m_0)\\ \nu(D) = k}}} \#\{a \in \fq[t]^{<m} : D^2 \mid f(a)\}$$
so that $\#N' = \sum_{k=0}^{\infty} (-1)^k n_k$. Brun's sieve is essentially the observation
that the partial sums $N_r = \sum_{k=0}^{r} (-1)^k n_k$ alternate around the limit $\#N'$, i.e.
$\#N' \le N_r$ for all even $r$, and $\#N' \ge N_r$ for all odd $r$ \cite[Chapter 6]{CoMu}. It will therefore
suffice to prove that $N_r = c_f q^m (1+o(1))$ for sufficiently large $r$, which will then
result in both upper and lower bounds on $\#N'$.

Suppose $m_0,r$ satisfy $2m_0r \le m$. It follows that for any $D \mid P(m_0)$ with $\nu(D) \le r$ 
we have $\deg(D^2) < 2m_0r \le m$. Such $D$ then satisfies 
$$\#\{a \in \fq[t]^{<m} : D^2 \mid f(a)\} = \rho(D^2)q^{m-2\deg D}.$$
Therefore for all $k \le r$, we have 
$n_k = \sum_{D | P(m_0), \nu(D) = k}\rho(D^2)q^{m-2\deg D}$, hence
$$N_r = q^m \sum_{{\substack{D \mid P(m_0)\\ \nu(D) \le r}}} 
\mu(D) \frac{\rho(D^2)}{||D||^2} =: q^m U(r,m_0). $$

We now wish to estimate $U(r,m_0)$. Note that 
\begin{align*}
  U(\infty,m_0) &  = \sum_{D \mid P(m_0)} \mu(D) \frac{\rho(D^2)}{||D||^2}
   = \prod_{P \in \PR^{<m_0}} \left(1-\frac{\rho(P^2)}{||P||^2}\right) \\
  & = c_f  \prod_{P \in \PR^{\ge m_0}} \left(1-\frac{\rho(P^2)}{||P||^2}\right)^{-1}
  = c_f(1+O(S(m_0)) = c_f(1+o(1)),
\end{align*}
where in the last step we assume $m_0$ is chosen such that $S(m_0) = o(c_f)$, as was already
required for bounding $\#N''$, so in particular $S(m_0) = o(1)$.

It will thus suffice to bound $U(\infty,m_0) - U(r,m_0)$. Let us denote for any non-negative
integer $k$, $v_k = \sum_{D | P(m_0), \nu(D) = k}\frac{\rho(D^2)}{||D||^2}$. Note 
that $v_k$ is the $k$-th elementary symmetric polynomial of the 
finite multiset $\left\{\frac{\rho(P^2)}{||P||^2} : P \in \PR^{<m_0}\right\}$, whose elements
are positive real numbers. It follows that $v_k \le \tfrac{v_1^k}{k!}$. \linebreak
Furthermore $v_1$ is
a partial sum of the singular sum $S$, hence $v_1 \le \lambda =  k \ln\log_q n + O(1)$ 
by \eqref{S bound}. Suppose $r = \alpha \lambda$ for some $\alpha > 2$. Then

\begin{align*}
 &\ |U(\infty,m_0) - U(r,m_0)| = \left|\sum_{k=r+1}^{\infty} (-1)^k v_k\right| 
 \le  \sum_{k=r+1}^{\infty} v_k \le \sum_{k=r+1}^{\infty} \frac{\lambda^k}{k!} \\
 < &\ \sum_{k=r+1}^{\infty} \frac{\lambda^r}{r!}\alpha^{r-k} < \frac{\lambda^r}{r!}
 < \frac{\lambda^r}{(r/e)^r} = \left(\frac{e\lambda}{r}\right)^r 
 = \left(\frac{e}{\alpha}\right)^{\alpha \lambda} \\
 = &\ O\left((\log_q n)^{-\alpha\ln(\alpha/e) k}\right).
\end{align*}
\iftoggle{omegas}
{Now if $\alpha = \omega(1)$ then 
$|U(\infty,m_0) - U(r,m_0)| = (\log_qn)^{-k\omega(1)} = o(c_f)$.}
{Now if $\alpha \ln(\alpha/e)$ is sufficiently large\footnote{In the case
$n \rinf$ it suffices to choose $\alpha \ln(\alpha/e) > 2$, 
which holds for $\alpha > 4.32$. If $n$ is bounded, take $\alpha \rinf$.
}, then by \eqref{c_f bound},
$$|U(\infty,m_0) - U(r,m_0)| \ll (\log_qn)^{-k\alpha\ln(\alpha/e)} = o(c_f).$$}

We have thus shown that for 
\iftoggle{omegas}
{$r = \omega(\log_q\log_qn)$}
{sufficiently large $r$ satisfying $r \gg \log_q\log_q n$ and $r \rinf$},
$N_r =  q^m c_f(1+o(1))$, hence also $\#N' =c_f q^m(1+o(1))$, as claimed.

For the proofs of the bounds on $N', N''$ to be valid simultaneously, we must be able to choose
\iftoggle{omegas}
{$m_0 = \omega(\log_qn)$ and $r = \omega(\log_q\log_qn)$ 
with $2m_0r \le m$, which is of course possible if and only if
$m = \omega(\log_qn \log_q\log_qn)$,} 
{$m_0,r$ with $m_0 \gg \log_qn$, $r \gg \log_q\log_qn $, $m_0,r \rinf$
and $2m_0r \le m$. This is of course possible if and only if
$m \gg \log_qn \log_q\log_qn$ and $m \rinf$,}
hence our condition on $m$ in Theorem \ref{brun thm}. Careful examination
of the required lower bounds on $r, m_0$ would allow the constant
in the constraint $m \gg \log_q \log_q\log_qn$ to be as small as $9k\ln q$
for sufficiently large $n$. 

\subsection{Bounding $N'''$: Large primes}

The large primes require the most sophistication to estimate, though they contribute the smallest error. To do so, we apply
Poonen's technique of replacing our target polynomial by an equivalent multivariate polynomial with a simpler $t$-derivative, and carefully
retrace Lando's bounds on the corresponding contributions to $N'''$, noting the size of our coefficients.

Given the polynomial $f(x) \in \fq[t][x]$, we define a new polynomial $F$ by
$F(y_0,\dots,y_{p-1}) = f(y_0^p + ty_1^p + \cdots + t^{p-1}y_{p-1}^p) \in 
\fq[t][y_0^p,y_1^p,\dots,y_{p-1}^p].$
 Note that $\deg_x(f) \le k, \deg_t(f) \le n$ together imply a bound on $F$'s coefficients and degrees:
$\deg_t(F) < n + pk = O(n)$, $\deg_{y_i}(F) \le pk$.

Poonen's lemmas show that $f$ being square-free implies $F$ is, also \cite[Lemma 7.2]{Poonen}; which in turn implies that $F$ and
$G = \frac{\partial F}{\partial t}$ are coprime \cite[Lemma 7.3]{Poonen}\footnote{
Poonen in fact shows only that they are coprime in $\fq(t)[y_0,\dots,y_{p-1}]$,
whereas we need them to be coprime in $\fq[t][y_0,\dots,y_{p-1}]$. This is easy to verify --
it is enough to check that they have no common factor $P \in \fq[t]$. Such a factor 
will necessarily divide the contents of both $F(y_0,0,\dots,0) = f(y_0^p)$ and 
$G(y_0,0,\dots,0) = \frac{\partial f}{\partial t}(y_0^p)$. This in turn implies
that $P^2$ divides $f$, contradicting our assumption
that it is square-free.}. On the other hand, for any $y \in (\fq[t])^p$, $P^2 \mid F(y)$ 
if and only if $P \mid F(y)$ and $P \mid G(y)$. This is due to the fact that, as the $y_i$-s appear in $F$ only with exponents
divisible by $p$, $G(y) = \frac{\d(F(y))}{\d t}$ for all $y$.
Finally observe that $\deg_t G \le \deg_t F = O(n)$,  $\deg_{y_i}(G) \le \deg_{y_i}(F) \le pk$.

Let $m_p = \ceil{\frac{m}{p}} \le \ceil{\frac{m}{2}} = m_1$, and for any positive integer $l$, let $B_l = (\fq[t]^{<m_p})^{l+1}$.
Note that when we let the $p$-tuple
$y$ range over all $B_{p-1}$, $a = y_0^p + ty_1^p + \cdots + t^{p-1}y_{p-1}^p$ ranges over all $\fq[t]^{<pm_p}$, which
contains $\fqtm$. Thus

\begin{align}
 \nonumber \#N''' & = \#\{a \in \fqtm : \exists P \in \PR^{\ge m_1}, P^2 \mid f(a)\} \\
 \nonumber & \le \#\{y \in B_p : \exists P \in \PR^{\ge m_1}, P^2 \mid f(y_0^p + ty_1^p + \cdots + t^{p-1}y_{p-1}^p)\} \\
 \nonumber & = \#\{y \in B_p : \exists P \in \PR^{\ge m_1}, P^2 \mid F(y)\} \\
 \nonumber & = \#\{y \in B_p : \exists P \in \PR^{\ge m_1}, P \mid F(y) \text{ and } P \mid G(y) \} \\
 \label{N''' est}
 & = O_{p,pk}\left(\frac{n+m_1}{m_1}q^{(p-1)m_p}\right)
 = O_{p,k}\left(\frac{n+m}{mq^{\tfrac mp - p}}q^m\right),
\end{align}
where the bound in the final line follows from the following proposition, analogous to \cite[Proposition 5]{Lando}:
\begin{prop}\label{B_k bound}
 Let $k,l,n,m_p,m_1$ be positive integers with $m_1 \ge m_p$, let $f,g \in \fq[t][y_0,\dots,y_l]$ be coprime polynomials
 in $l+1$ variables with
 $\deg_{y_i}(f), \deg_{y_i}(g)\le k$ and $\deg_t(f), \deg_t(g) \le n$, and $B_l$ as above.
 Define $$\NC_l(f,g) = \#\{y \in B_l : \exists P \in \PR^{\ge m_1}, P \mid f(y) \text{ and } P \mid g(y) \}.$$
 Then $\NC_l(f,g) = O_{l,k}\left(\frac{n+m_1}{m_1}  q^{lm_p}\right)$.
\end{prop}

Thus, from \eqref{N''' est} and \eqref{c_f bound}, it follows that $\#N''' = o(c_f q^m)$ when e.g.
\iftoggle{omegas}
{$m = p(\log_qn + 2k \log_q \log_qn) +\omega(1)$,}
{$m - p(\log_qn + 2k \log_q \log_qn) \rinf$,}
which is certainly the case under the assumptions 
of Theorem \ref{brun thm}. 

Before we prove proposition \ref{B_k bound}, we first need a simpler bound,
slightly generalizing \cite[Proposition 6]{Lando} and giving exact bounds.
\begin{prop}\label{B_k 0 bound}
 Let $k,l,n,m_p,f,B_l$ be as in Proposition \ref{B_k bound}, and suppose $f$ is not identically $0$. Then
 $$\#\{y \in B_l : f(y) = 0\} \le k(l+1)q^{lm_p}.$$
\end{prop}
\begin{proof}
 If $l = 0$, then $f(y_0)$ is a non-vanishing polynomial of degree at most $k$ in $y_0$. Hence it has at most $k$ roots 
 in all of $\fq[t]$, and in particular $\#\{y \in B_0 : f(y) = 0\} \le k$, as claimed.
 
 We proceed by induction on $l$.
 Consider $f$ as a polynomial in $y_l$, of degree at most $k$, with coefficients in $\fq[t][y_0,\dots,y_{l-1}]$.
 We write it as $f(y',y_l)$, where $y' = (y_0,\dots,y_{l-1})$. 
 Let $f_0 \in \fq[t][y_0,\dots,y_{l-1}]$ be its leading coefficient. Clearly, $f_0$ also satisfies the degree requirements
 of Proposition \ref{B_k 0 bound}, hence by induction,
 \begin{equation}\label{B_k 0 induction}
 \#\{y' \in B_{l-1} : f_0(y') = 0\} \le klq^{(l-1)m_p}.
 \end{equation}
 On the other
 hand, for any $y' \in B_{l-1}$ with $f_0(y) \neq 0$, there are at most $\deg_{y_l}(f) \le k$ values of $y_l$ in all $\fq[t]$
 for which $f(y',y_l) = 0$. Thus 
 \begin{equation}\label{B_k 0 newterm}
 \#\{(y',y_l) \in B_l : f_0(y') \neq 0, f(y',y_l) = 0\} \le k\#B_{l-1} = kq^{lm_p}. 
 \end{equation}
 
 Using both \eqref{B_k 0 induction}, \eqref{B_k 0 newterm}, we finally obtain
 \begin{align*}
  &\#\{(y',y_l) \in B_l : f(y',y_l) = 0\} \\ 
  \le\ &\#\{(y',y_l) \in B_l : f_0(y') = 0\} + \#\{(y',y_l) \in B_l : f_0(y') \neq 0,f(y',y_l) = 0\} \\
   =\ &q^{m_p}\#\{y' \in B_{l-1} : f_0(y') = 0\} + \#\{(y',y_l) \in B_l : f_0(y') \neq 0,f(y',y_l) = 0\} \\
   \le\ &q^{m_p}klq^{(l-1)m_p} + kq^{lm_p} = k(l+1)q^{lm_p}. 
 \end{align*}
\end{proof}

Using exactly the same arguments, one may also show the following 
similar proposition:
\begin{prop} \label{B_k P bound}
 Let $k,l,n,m_p,m_1,f,B_l$ be as in Proposition \ref{B_k bound}, let
 $P \in \PR^{\ge m_1}$ be a large
 prime and suppose $f$ is not identically $0$ modulo $P$. Then
 $$\NC_l(f,P) = \#\{y \in B_l : P \mid f(y)\} \le k(l+1)q^{lm_p}.$$
\end{prop}
Note that we rely strongly on $m_1 \ge m_p$, which implies that each 
residue class modulo $P$ has at most a single representative in $\fq[t]^{<m_p}$.
We omit the rest of the proof, which is just a repetition of the proof of
Proposition \ref{B_k 0 bound}.
\begin{proof}[Proof of Proposition \ref{B_k bound}]
  Again, we induce on $l$. To avoid repetition, our induction base will be $l=-1$,
  where $f,g \in \fq[t]$, and $B_{-1} = \{()\}$ is a singleton containing only the
  empty tuple. The claim then immediately follows from $f,g$ being coprime in $\fq[t]$,
  i.e. $\nexists P \in \PR$ such that $P \mid f$ and $P \mid g$, and in particular
  $\{y \in B_{-1} : \exists P \in \PR^{\ge m_1}, P \mid f(y) \text{ and } P \mid g(y) \}$ 
  is empty. Hence $\NC_l(f,g) = 0 = O_{k}(\frac{n+m_1}{m_1}  q^{-m_p})$.

  We denote $A_l = \fq[t,y_0,\dots,y_{l-1}]$. Consider $f,g \in A_l[y_l]$
  as single variable polynomials in $y_l$
  with coefficients in the polynomial ring $A_l$, and let 
  $f_C,g_C \in A_l$ be their respective contents. We may then write
  $f = f_Cf_I, g = g_Cg_I$ where $f_I,g_I \in A_l[y_l]$ are indivisible
  by any non-scalar polynomial in $A_l$. Clearly $f_C,f_I$ are coprime to 
  $g_C,g_I$, and all four polynomials have $y_i$-degrees at most $k$ and $t$-degrees at most $n$.
  We also have
  $$\NC_l(f,g) \le \NC_l(f_I,g_I) + \NC_l(f_I,g_C) + \NC_l(g_I,f_C) + \NC_l(f_C,g_C).$$
  Therefore it is enough to show that each of the four summands on the right hand side is
  bounded by $O_{l,k}(\frac{n+m_1}{m_1}q^{lm_p})$.
  
  Note that, as both $f_C$ and $g_C$ are independent of $y_l$, and by the induction hypothesis,
  we have 
  \begin{align*}
  \NC_l(f_C,g_C) = q^{m_p}\NC_{l-1}(f_C,g_C) 
  &= q^{m_p}O_{l-1,k}\left(\frac{n+m_1}{m_1}q^{(l-1)m_p}\right) \\
  &= O_{l,k}\left(\frac{n+m_1}{m_1}q^{lm_p}\right).  
  \end{align*}
  
  For both $\NC_l(f_C,g_I)$, $\NC_l(f_I,g_C)$, we have one polynomial in $A_l$ and the second
  indivisible by any polynomial in $A_l$. We wish to bound $\NC_l(f_I,g_I)$ by a term of this
  form as well. To do so, let $R = \rest_{y_l}(f_I,g_I) \in A_l$ be the resultant of $f_I,g_I$.
  By basic properties of the resultant, for any choice of $y_i \in \fq[t], P \in \PR$,
  we have $P \mid f_I(y), P \mid g_I(y) \implies P \mid R(y)$.
  Thus $\NC_l(f_I,g_I) \le \NC_l(f_I,R)$.
  Further note that from $\deg_{y_l}(f_I),\deg_{y_l}(g_I) \le k$ it follows that 
  $R$ is given as a polynomial of degree $\le 2k$ in the $A_l$ coefficients of $f_I,g_I$,
  Hence in particular $\deg_{t}(R) \le 2kn$, $\deg_{y_i}(R) \le 2k^2$. 
  Also note that $R$ is non-zero, as $f_I,g_I$ are co-prime.
  
  We now claim that for any polynomials $R \in A_l, f \in A_l[y_l]$ such that $f$ is
  indivisible by non-scalar polynomials in $A_l$, and with $\deg_t f \le n, \deg_{y_i} f \le k$
  and $\deg_t R \le 2kn, \deg_{y_i} R \le 2k^2$, we have 
  $\NC_l(f,R) = O_{l,k}(\frac{n+m_1}{m_1}q^{lm_p})$. 
  This bound would then be applicable to $\NC_l(f_I,g_C)$, $\NC_l(g_I,f_C)$ and $\NC_l(f_I,g_I)$,
  finishing our induction step.

  Let $R = \prod_{j \in J} R_j$ be $R$'s decomposition into irreducible polynomials.
  We have $\NC_l(f,R) \le \sum_{j \in J} \NC_l(f,R_j)$. 
  Note that for each $j$, $R_j \in A_l$, therefore $R_j \nmid f$ and $f,R_j$ are coprime.
  Let us partition $J = J_1 \cup J_2 \cup J_3$, where $J_1 = \{j \in J : R_j \notin \fq[t]\}$,
  $J_2 = \{j \in J : R_j \in \fq[t]^{\ge m_1}\}$,
  and  $J_3 = \{j \in J : R_j \in \fq[t]^{< m_1}\}$.
  As $\deg_t R \le 2kn$ and the total degree of $R$ in all  $y$-variables is at most $2k^2l$, 
  we have $\#J_2 \le \frac{2kn}{m_1}, \#J_1 \le 2k^2l$.
  
  For each $j \in J_3$, $y \in B_l$, we have $R_j(y) = R_j$, so clearly 
  $\nexists P \in \PR^{\ge m_1}$  with $P \mid R_j$, hence $\NC_l(f,R_j) = 0$.
  Similarly, for each $j \in J_2$, the conditions of proposition \ref{B_k P bound}
  are satisfied for $f,P=R_j$. Hence $\NC_l(f,R_j) \le k(l+1)q^{lm_p} = O_{l,k}(q^{lm_p})$.

  Finally, for each $j \in J_1$, let $f_0 \in A_l$ be some 
  coefficient of $f$ (as a polynomial in $y_l$) such that $R_j \nmid f_0$ as polynomials.
  Such a coefficient must exist as $R_j \nmid f$. We now bound $\NC_l(f,R_j)$,
  again by splitting into three trivially covering sets:
  \begin{align*}
   & \NC_l(f,R_j)  = \#\{y \in B_l : \exists P \in \PR^{\ge m_1}, P \mid f(y) \text{ and } P \mid R_j(y)\} \\
    & \le  \#\{y \in B_l : R_j(y) = 0\} \\
    & + \#\{y \in B_l : \exists P \in \PR^{\ge m_1}, P \mid f_0(y) \text{ and } P \mid R_j(y)\} \\
    & + \#\{y \in B_l : R_j(y) \neq 0, \exists P \in \PR^{\ge m_1}, P \mid R_j(y), P \mid f(y)  \text{ and } P \nmid f_0(y)\}.
  \end{align*}
  
  By proposition \ref{B_k 0 bound}, the first summand is clearly $O_{l,k}(q^{lm_p})$.
  The second summand, by definition, is $\NC_l(f_0,R_j)$.  As $R_j$ is irreducible,
  it follows that $f_0, R_j$ are coprime. We also certainly have 
  $\deg_{y_i}(f_0),\deg_{y_i}(R_j) \le 2k^2$ and 
  $\deg_t(f_0),\deg_t(R_j) \le 2kn$. Therefore $f_0,R_j$ satisfy the conditions 
  of proposition \ref{B_k bound}, but with smaller $l$ (albeit larger degrees).
  Hence by the induction hypothesis, 
  \begin{align}
   \nonumber \NC_l(f_0,R_j) = q^{m_p}\NC_{l-1}(f_0,R_j)
   & = q^{m_p}O_{l-1,2k^2}\left(\frac{2kn+m_1}{m_1}q^{(l-1)m_p}\right)
  \\ & = O_{l,k}\left(\frac{n+m_1}{m_1}q^{lm_p}\right).
  \end{align}
  
  To bound the third term, note that for each $y = (y',y_l) \in B_{l-1} \times B_0 =  B_l$ such that
  $R_j(y) = R_j(y') \neq 0$, we must have $\deg_t(R_j(y')) \le 2kn + 2k^2lm_p$.
  If we let $\PR_{y'} = \{P \in \PR^{\ge m_1} : P \mid R_j(y'), P \nmid f_0(y')\}$, it follows that
  $\#\PR_{y'} \le \frac{2kn+2k^2lm_p}{m_1} = O_{l,k}(\frac{n+m_1}{m_1})$.
  On the other hand, for each $y' \in B_{l-1}, P \in \PR_{y'}$, 
  $f(y',y_l)$ is a polynomial of degree $\le k$ in $y_l$, which is non-vanishing
  modulo $P$. Since $\deg_t(P) \ge m_1 \ge m_p$, it follows that 
  $\#\{y_l \in B_0 : P \mid f(y',y_l)\} \le k$. Therefore
  \begin{align*}
   &\#\{y \in B_l : R_j(y) \neq 0, \exists P \in \PR^{\ge m_1}, P \mid R_j(y), P \mid f(y)  \text{ and } P \nmid f_0(y)\}
   \\ = &\sum_{\substack{y' \in B_{l-1}\\ R_j(y') \neq 0}} 
           \#\{y_l \in B_0:  \exists P \in \PR^{\ge m_1}, P \mid R_j(y'), P \mid f(y',y_l)  \text{ and } P \nmid f_0(y')\}
   \\ \le &\sum_{\substack{y' \in B_{l-1}\\ R_j(y') \neq 0}}
           \sum_{P \in \PR_{y'} } \#\{y_l \in B_0 : P \mid f(y',y_l)\}
      \le \sum_{\substack{y' \in B_{l-1}\\ R_j(y') \neq 0}}
           \sum_{P \in \PR_{y'} } k
   \\ = &\sum_{\substack{y' \in B_{l-1}\\ R_j(y') \neq 0}} O_{l,k}\left(\frac{n+m_1}{m_1}\right)
      =  O_{l,k}\left(\frac{n+m_1}{m_1} q^{lm_p}\right).
  \end{align*}
  
  Taking the three results together, we find $\NC_l(f,R_j) = O_{l,k}(\frac{n+m_1}{m_1} q^{lm_p})$
  for all $j \in J_1$. Now combining the different bounds for each $J_i$, we finally obtain
  
  \begin{align*}
   \NC_l(f,g) & \le \sum_{j \in J_1} \NC(f,R_j) + \sum_{j \in J_2} \NC(f,R_j) + \sum_{j \in J_3} \NC(f,R_j)
   \\ & = \sum_{j \in J_1}  O_{l,k}\left(\frac{n+m_1}{m_1} q^{lm_p}\right) + 
            \sum_{j \in J_2} O_{l,k}\left(q^{lm_p}\right) + \sum_{j \in J_3} 0
   \\ & \le 2k^2l \cdot O_{l,k}\left(\frac{n+m_1}{m_1} q^{lm_p}\right) + 
            \frac{2kn}{m_1} \cdot O_{l,k}\left(q^{lm_p}\right)
   \\ & = O_{l,k}\left(\frac{n+m_1}{m_1} q^{lm_p}\right),
  \end{align*}
  as we wanted to show.

\end{proof}

\section{Proof of Theorems \ref{short interval thm}, \ref{unif thm}}
\subsection{Proof of Theorem \ref{short interval thm}}
Define $f(x) = g(t,N(t)+x) \in \fq[t][x]$. Clearly $\deg_x f = \deg_x g = k$, and
$\deg_t f \le \deg_x g \cdot \deg_t N + \deg_t g = kn + \deg_t g = O(n)$. Furthermore,
as $f$ is obtained from $g$ simply by a fixed $\fq[t]$ translation of the $x$ variable,
$g$ being square-free implies that $f$ is square-free, and more importantly,
$\rho_f(D) = \rho_g(D) = \rho(D)$ for any polynomial $D$. Therefore they also have
the same singular sum and series, i.e. $S_f(m_0) = S_g(m_0)$, as well as $S_f = S_g$ and $c_f = c_g$
being constants independent of the choice of $N(t)$ or its degree $n$. Thus taking any
\iftoggle{omegas}
{$m_0 = \omega(1), r = \omega(1)$}
{$m_0 \rinf, r \rinf$}, we have immediately
$S(m_0) = o(1) = o(c_f)$ and $\frac{r}{S} \rinf$, from which we obtain $\#N' = c_f q^m (1+o(1))$
and $\#N'' = o(c_f q^m)$ following the proofs in sections \ref{N' sect}, \ref{N'' sect}.
To be able to choose such $r,m_0$, we only need 
\iftoggle{omegas}
{$m = \omega(1)$}
{$m \rinf$}.

We are left only with the need to validate the bound on $N'''$, and here finally $n$ does 
come into play, as it still affects the relevant degrees. As $c_f$ is now a constant,
\eqref{N''' est} implies that $\#N''' = o(1) = o(c_f)$ when 
\iftoggle{omegas}
{$mq^{m/p} = \omega(n)$, which is equivalent to 
$m = p(\log_qn - \log_q\log_qn) + \omega(1)$,}
{$\frac{mq^{m/p}}{n} \rinf$, which is equivalent to 
$m - p(\log_qn - \log_q\log_qn) \rinf$,}
as we required in the theorem's statement. \qed

\subsection{Proof of Theorem \ref{unif thm}}
Similarly to the above, we observe that when we move to $f(x) = g(N(t)+x)$, the expressions
determined by the singular sum, $S,S(m_0)$ and $c_f$, will depend only on $g$ and not on $N$.
Thus the bounds \eqref{S bound}--\eqref{c_f bound} will all be valid with $n$ replaced by $n_1$,
as will the computations of sections \ref{N'' sect}, \ref{N' sect}, as long as we may choose
$r,m_0 \rinf$ with 
\iftoggle{omegas}
{$r = \omega(\log_q\log_q n_1)$, $m_0 = \omega(\log_q n_1)$ and $2m_0r \le m$,
which can be satisfied due to $m = \omega(\log_qn_1 \log_q\log_qn_1)$.}
{$m_0 \gg \log_q n_1$, $r \gg \log_q\log_q n_1$ and $2m_0r \le m$,
which is possible due to the assumption $m \gg \log_qn_1 \log_q\log_qn_1$.}

For the bound on $\#N'''$, we observe that $\deg_t f \le kn_2 + n_1$. If $n_2 \ll n_1$, then
$\deg_t f \ll n_1$ and we are basically in the case of Theorem \ref{brun thm}, where the
contribution of $N'''$ is negligible. Otherwise, $n_2$ is much greater than $n_1$, so 
$\deg_t f \ll n_2$. Thus \eqref{N''' est} holds with the degree $n$ replaced by $n_2$. Taken 
together with \eqref{c_f bound} with $n$ replaced by $n_1$, we
see that $\#N''' = o(c_fq^m)$ would follow from 
\iftoggle{omegas}
{$mq^{m/p} = \omega(n_2 (\log_q n_1)^{2k})$, which is equivalent to 
$m = p(\log_q n_2 - \log_q \log_q n_2 + 2k\log_q\log_qn_1) + \omega(1)$,}
{$\frac{mq^{m/p}}{n_2 (\log_q n_1)^{2k}} \rinf$, which is equivalent to 
$m - p(\log_q n_2 - \log_q \log_q n_2 + 2k\log_q\log_qn_1) \rinf$,}
as we required. \qed

\begin{remark}
We can in fact make a slight improvement here on the required condition: By using 
$c_f \gg (\log_q n_1)^{-k - o(1)}$ instead of \eqref{c_f bound}, 
the constant coefficient $2k$ can be replaced with any constant greater than $k$, or with some (specific)
function of the type $k + o(1)$.
\end{remark}

\section*{Acknowledgements}

The author wishes to thank Ze\'ev Rudnick for proposing the problems which motivated
the majority of this paper, as well as many helpful suggestions and references. The author
also thanks Alexei Entin, for contributing the appendix, and for suggesting the 
use of Brun's sieve, without which the results of this paper would have been far weaker.



\newpage 

\appendix


\edef\shortauthors{Alexei Entin}
\edef\authors{Alexei Entin}
\title[Squarefree Integers in Short Intervals]
{Appendix: On the Number of Squarefree Integers in Short Intervals}

\date{}
\edef\thankses{}

\maketitle
\begin{abstract} Assuming the ABC conjecture we show that for any fixed
$\eps>0$ the number of squarefree integers in the interval $\lbb x,x+H\rb$ is
$\sim \frac{6}{\pi^2}H$ provided $H>x^\eps$. 

\end{abstract}
\refstepcounter{section}
\subsection{Introduction}

We consider the problem of counting the number of squarefree integers in the 
interval $\lbb x,x+H\rbb$, where $x$ and $H$ are large positive real numbers.
We are interested in the case that $H=x^\eps$ for some fixed $\eps>0$ while
$x\to\ity$. It is an open problem to show that for any fixed $\eps>0$ there
exists even a single squarefree integer in the interval $\lbb x,x+H\rbb$ with
$H=x^\eps$ for large enough $x$. The best known result in this direction is due
to Filaseta and Trifonov \cite{Filatri} who showed the existence of squarefree
integers in $\lbb x,x+H\rbb$ for $H\gg x^{1/5}\log x$. It was shown by Tolev
\cite{Tolev} 
that when $\frac{H}{x^{1/5}\log x} \to\ity$, the number of squarefrees in the
interval $\lbb x,x+H\rbb$ is in fact asymptotic to $(6/\pi^2)H$.
It was shown by Granville \cite{Granville} that assuming the ABC conjecture 
for any fixed $\eps>0$ there exists a squarefree integer in 
$\lbb x,x+x^\eps\rbb$ for $x$ large enough. Our main result is the following:

\begin{thm}\label{main} Assume the ABC conjecture. Let $\eps>0$ be fixed.
Then the number of squarefree integers in the interval $\lbb x,x+H\rb$ is
$\sim \frac{6}{\pi^2}H$ provided $H>x^\eps$.\end{thm}

We note that $6/\pi^2=\ze(2)^{-1}$, where $\ze(s)$ is the Riemann zeta-function.
By essentially the same argument it can be shown that assuming the ABC
conjecture for any fixed $k$ the number of $k$-power-free integers in
$\lbb x,x+H\rb$ is $\sim \ze(k)^{-1}H$ provided $H>x^\eps$ for fixed $\eps>0$.

%
%

\subsection{Proof of Theorem \ref{main}}

\begin{prop}\label{sievesmall} The number of integers in the interval 
$\lbb x,x+H\rb$ which are not divisible by any square of a prime $p<H$ 
is $\sim \frac{6}{\pi^2}H$ as $H\to\ity$.\end{prop}

\begin{proof} It is elementary to see that the number of integers in
$\lbb x,x+H\rb$ not divisible by $p^2$ for any $p<\frac{1}{2}\log H$
is $\sim \ze(2)^{-1}H=\frac{6}{\pi^2} H$ (this is seen by exact sieving over
all primes up to $\frac{1}{2}\log H$). The number of integers in
$\lbb x,x+H\rb$ divisible by $p^2$ for some $\frac{1}{2}\log H<p<H$ is bounded by 
$$\sum_{\frac{1}{2}\log H<p<H}\lb \frac{H}{p^2}+1\rb\ll \frac{H}{\log H}=o(H),$$
which is asymptotically negligible.\end{proof}

We will need the following result due to Granville \cite[Corollary 2.1]{Granville}:

\begin{prop}\label{propgran} Assume the ABC conjecture. 
Let $F(X)\in\Z[x]$ be a fixed squarefree polynomial and $\al>0$ a fixed 
constant. Let $y$ be a natural number and assume that $s^2|F(y)$ for some
natural number $s$. Then for $y$ large enough we have
$s\le y^{1+\al}$.\end{prop}

\begin{prop}\label{sievelarge} Assume the ABC conjecture.
If $H<x$ and $H\to\ity$ then the number of integers in $\lbb x,x+H\rb$
divisible by the square of any prime $p>x^\eps$ is $o(H)$.\end{prop}

\begin{proof} Let $\lam>0$ be a constant. Assume that the number of integers
in $\lbb x,x+H\rb$ divisible by $p^2$ for some prime $p>x^\eps$ is $>\lam H$.
We want to show that $H$ must be bounded (for any fixed $\lam$). Denote 
$N=\ceil{2/\eps}, M=\ceil{2N/\lam}$ (these are both fixed constants for
fixed $\eps,\lam$). The interval $\lbb x,x+H\rb$ necessarily contains a 
subinterval $\lbb y,y+M\rb$ with at least $\frac{1}{2}\lam M\ge N$ (if 
$M$ divides $H$ the $\frac{1}{2}$ factor is unnecessary) elements divisible
by some $p^2$ for some prime $p>x^\eps\gg y^\eps$.

Assuming by way of contradiction that $H$ can be arbitrarily large, we see
that there must exist arbitrarily large $y$ s.t. at least $N$ integers in
the interval $\lbb y,y+M\rb$ are divisible by a square of some prime
$p\gg y^\eps$. By the pigeonhole principle there must exist some fixed 
distinct $a_1,...,a_N\ge 0$ s.t. for infinitely many $y$ each $y+a_1,...,y+a_N$
is divisible by the square of some prime $p\gg y^\eps$.

Denote $F(X)=(X+a_1)...(X+a_N)\in\Z[x]$. This is a squarefree polynomial.
From the above we see that for infinitely many $y$ the value $F(y)$ is
divisible by the square of some $d=p_1...p_N\gg y^{N\eps}\ge y^2$. But this
contradicts Proposition \ref{propgran} (taking any $\al<1$ in the proposition).
\end{proof}

Combining Proposition \ref{sievesmall} and Proposition \ref{sievelarge}
we deduce Theorem 1.



\newpage

\edef\shortauthors{Dan Carmon}
\edef\authors{Dan Carmon}
\author{Dan Carmon}
\title[Square-free values of large polynomials]
{On square-free values of large polynomials 
over the rational function field}
\refreshheaders

\bibliography{large_coefs}{}
\bibliographystyle{hacm}
\setaddresses

\end{document}